\documentclass{amsart}
\usepackage{amssymb}
\usepackage{enumerate}
\usepackage{url}

\newcommand\N{\mathbb N}

\newcommand\R{\mathbb R}

\newcommand\ph\varphi
\newcommand\ps\psi
\newcommand\ep\varepsilon
\newcommand\rh\varrho
\newcommand\al\alpha
\newcommand\be\beta
\newcommand\ga\gamma
\newcommand\om\omega
\newcommand\ta\tau
\renewcommand\th\vartheta
\newcommand\de\delta
\newcommand\ze\zeta
\newcommand\ch\chi
\newcommand\et\eta
\newcommand\io\iota
\newcommand\la\lambda
\newcommand\si\sigma

\newcommand\Ga\Gamma
\newcommand\De\Delta
\newcommand\Th\Theta
\newcommand\La\Lambda
\newcommand\Si\Sigma
\newcommand\Ph\Phi
\newcommand\Ps\Psi
\newcommand\Om\Omega

\newtheorem{theorem}{Theorem}
\newtheorem{lemma}[theorem]{Lemma}
\newtheorem{proposition}[theorem]{Proposition}
\newtheorem{corollary}[theorem]{Corollary}
\theoremstyle{definition}
\newtheorem{definition}[theorem]{Definition}

\theoremstyle{remark}
\newtheorem{remark}[theorem]{Remark}

\newcommand\x{{\bar X}}
\newcommand\rx{{\R[\x]}}
\newcommand\sos{{\sum\rx^2}}
\newcommand{\dist}{\text{dist}}
\newcommand{\diam}{\text{diam}}

\begin{document}
\title[Putinar's Positivstellensatz]
{On the complexity of Putinar's Positivstellensatz}
\author{Jiawang Nie}
\thanks{The first author is supported by National Science Foundation
DMS-0456960.}
\address{Department of Mathematics\\
         University of California\\
         Berkeley, CA 94720-3840}
\email{njw@math.berkeley.edu}
\author{Markus Schweighofer}
\thanks{The second author is supported by the DFG grant ``Barrieren''.}
\address{Fachbereich Mathematik und Statistik\\
         Universit\"at Konstanz\\
         78457 Konstanz\\
         Germany}
\email{Markus.Schweighofer@uni-konstanz.de}
\date{\today}

\begin{abstract}
Let $S=\{x\in\R^n\mid g_1(x)\ge 0,\dots,g_m(x)\ge 0\}$ be a basic closed
semialgebraic set defined by real polynomials $g_i$. Putinar's
Positivstellensatz says that, under a certain condition stronger than
compactness of $S$, every real polynomial $f$ positive on $S$ posesses a
representation $f=\sum_{i=0}^m\si_ig_i$ where $g_0:=1$ and each $\si_i$ is a
sum of squares of polynomials. Such a representation is a certificate for the
nonnegativity of $f$ on $S$. We give a bound on the degrees of
the terms $\si_ig_i$ in this representation which depends on the description
of $S$, the degree of $f$ and a measure of how close $f$ is to having a zero
on $S$. As a consequence, we get information about the convergence rate of
Lasserre's procedure for optimization of a polynomial subject to polynomial
constraints.
\end{abstract}

\maketitle

\section{Introduction}

Always write $\N:=\{0,1,2,\dots\}$ and $\R$ for the sets of nonnegative
integers and real numbers, respectively. Denote by $\rx$ the ring of
polynomials in $n\ge 1$
indeterminates $\x:=(X_1,\dots,X_n)$. We use suggestive notation like
$\rx^2:=\{p^2\mid p\in\rx\}$ for the set of squares and $\sum\rx^2$ for the
set of sums of squares of polynomials in $\rx$. A subset $M\subseteq\rx$ is
called
a \emph{quadratic module} if it contains $1$ and it is closed under
addition and under multiplication with squares, i.e.,
$$1\in M,\qquad M+M\subseteq M\qquad\text{and}\qquad\rx^2M\subseteq M.$$
A subset $T\subseteq\rx$ is called a \emph{preordering} if it
contains all squares in $\rx$ and it is closed under
addition and multiplication, i.e.,
$$\rx^2\subseteq T,\qquad T+T\subseteq T\qquad\text{and}\qquad
  TT\subseteq T.$$
In other words, the preorderings are exactly the multiplicatively
closed quadratic modules.

Throughout the article, we fix $m\in\N$ and a tuple $\bar g:=(g_1,\dots,g_m)$
of polynomials $g_i\in\rx$. It will be convenient to set $g_0:=1\in\rx$.
The quadratic module $M(\bar g)$ generated by
$\bar g$ (i.e., the smallest quadratic module containing each $g_i$) is
\begin{equation}\label{mm}
M(\bar g)=\sum_{i=0}^m\sum\rx^2g_i
:=\left\{\sum_{i=0}^m\si_ig_i\mid\si_i\in\sos\right\}.
\end{equation}
Using the notation
$${\bar g}^\de:=g_1^{\de_1}\dots g_m^{\de_m},$$
the preordering $T(\bar g)$ generated by $\bar g$ can be written as
\begin{equation}\label{tt}
T(\bar g)=\sum_{\de\in\{0,1\}^m}\sum\rx^2{\bar g}^\de:=
\left\{\sum_{\de\in\{0,1\}^m}\si_\de{\bar g}^\de\mid\si_\de\in\sos\right\},
\end{equation}
i.e., $T(\bar g)$ is the quadratic module generated by the $2^m$ products
of $g_i$. It is obvious that all polynomials lying in
$T(\bar g)\supseteq M(\bar g)$ are nonnegative on the set
$$S(\bar g):=\{x\in\R^n\mid g_1(x)\ge 0,\dots,g_m(x)\ge 0\}.$$
Sets of this form are important in semialgebraic geometry (see \cite{bcr})
and
are called \emph{basic closed semialgebraic sets}. In 1991,
Schm\"udgen \cite{smn} proved the following ``Positivstellensatz'' (a
commonly used German term explained by the analogy with Hilbert's
Nullstellensatz).

\begin{theorem}[Schm\"udgen]\label{schmuedgen}
Suppose the basic closed semialgebraic set $S(\bar g)$ is compact.
Then for every polynomial $f\in\rx$,
$$f>0\text{\ on\ }S(\bar g)\implies f\in T(\bar g).$$
\end{theorem}

Under a certain extra property on $M(\bar g)$ which we will define now,
this theorem remains true with $T(\bar g)$ replaced by its subset
$M(\bar g)$. We introduce the notation
$$\|\x\|^2:=\sum_{i=1}^n X_i^2\in\rx.$$
\begin{definition}
A quadratic module $M\subseteq\rx$ is called \emph{archimedean}
if $$N-\|\x\|^2\in M\qquad\text{for some\ }N\in\N.$$
\end{definition}

Note that this definition applies also to preorderings since every
preordering is a quadratic module. As a corollary from Schm\"udgen's Theorem,
we get the following well-known characterization of archimedean quadratic
modules.

\begin{corollary}\label{ac}
For a quadratic module $M\subseteq\rx$, the following are equivalent.
\begin{enumerate}[(i)]
\item $M$ is archimedean.\label{ac1}
\item There is a polynomial $p\in M$ such that
$S(p)=\{p\ge 0\}\subseteq\R^n$ is compact.\label{ac2}
\item There is a tuple $\bar g$ of polynomials such that $S(\bar g)$ is compact
and $M$ contains the preordering $T(\bar g)$.\label{ac3}
\item For all $p\in\rx$, there is $N\in\N$ such that $N-p\in M$.\label{ac4}
\end{enumerate}
\end{corollary}

\begin{proof}
Observe that (\ref{ac1})$\implies$(\ref{ac2})$\implies$(\ref{ac3})$\implies$
(\ref{ac4})$\implies$(\ref{ac1}). All of these implications are trivial
except (\ref{ac3})$\implies$(\ref{ac4}) which follows from Theorem
\ref{schmuedgen}.
\end{proof}

In particular, we see that $S(\bar g)$ is compact if and only if
$T(\bar g)$ is archimedean. Unfortunately, $S(\bar g)$ might be compact
without $M(\bar g)$ being archimedean (see \cite[Example 6.3.1]{pd}).
What has to be added to compactness of $S(\bar g)$ in order to ensure
that $M(\bar g)$ is archimedean has been extensively investigated by Jacobi
and Prestel \cite{jp,pd}. Now we can state the Positivstellensatz proved by
Putinar \cite{put} in 1993.

\begin{theorem}[Putinar]\label{putinar}
Suppose the quadratic module $M(\bar g)$ is archimedean. Then for
every $f\in\rx$,
$$f>0\text{\ on\ }S(\bar g)\implies f\in M(\bar g).$$
\end{theorem}

Both the proofs of Schm\"udgen and Putinar use functional analysis and real
algebraic geometry. They do not give information how to construct
a representation of $f$ showing that $f$ lies in the preordering
(an expression like in (\ref{tt}) involving $2^m$ sums of squares) or the
quadratic module (a representation like in (\ref{mm}) with $m+1$ sums of
squares).

Based on an old theorem of P\'olya \cite{pol}, new proofs of both
Schm\"udgen's and Putinar's Positivstellensatz have been given in
\cite{sw1,sw3}
which are to some extent constructive. By carefully analyzing a tame version
of \cite{sw3} and using an effective version of P\'olya's theorem
\cite{pr}, upper bounds on the degrees of the sums of
squares appearing in Schm\"udgen's preordering representation have been
obtained in \cite{sw2}. The aim of this article is to prove bounds on
Putinar's quadratic module representation. They will depend on the same data
but will be worse than the ones known for Schm\"udgen's theorem.

Since it will appear in our bound, we will need a convenient measure
of the size of the coefficients of a polynomial. For $\al\in\N^n$, we
introduce the notation
$$|\al|:=\al_1+\dots+\al_n\qquad\text{and}\qquad\x^\al:=X_1^{\al_1}\dotsm
  X_n^{\al_n}$$
as well as the multinomial coefficient
$$\binom{|\al|}{\al}:=\frac{|\al|!}{\al_1!\dots\al_n!}.$$
For a polynomial $f=\sum_\al a_\al\x^\al\in\rx$ with coefficients
$a_\al\in\R$, we set
$$\|f\|:=\max_\al\frac{|a_\al|}{\binom{|\al|}{\al}}.$$
This defines a norm on the real vector space $\rx$ with convenient properties
illustrated by Proposition \ref{normprop} below. For any $k\in\R_{\ge 0}$, we
now define convex cones $T(\bar g,k)$ and $M(\bar g,k)$ in the
finite-dimensional vector space $\rx_{\le k}$ of polynomials of degree at
most $k$ (i.e., at most $\lfloor k\rfloor$) by setting
\begin{align*}
T(\bar g,k)&=\left\{\sum_{\de\in\{0,1\}^m}\si_\de{\bar g}^\de\mid
\si_\de\in\sos,\deg(\si_\de{\bar g}^\de)\le k\right\}&\subseteq T(\bar g)
\cap\rx_{\le k},\\
M(\bar g,k)&=\left\{\sum_{i=0}^m\si_\de{\bar g}^\de\mid
\si_\de\in\sos,\deg(\si_\de{\bar g}^\de)\le k\right\}
&\subseteq M(\bar g)\cap\rx_{\le k}
\end{align*}
We now recall the previously proved bound for Schm\"udgen's theorem.
\begin{theorem}[\cite{sw2}]\label{schmuedgenbound}
For all $\bar g$ defining a basic closed semialgebraic set $S(\bar g)$
which is non-empty and contained in the open hypercube $(-1,1)^n$,
there is some $c\ge 1$ (depending on $\bar g$) such that for
all $f\in\rx$ of degree $d$ with
$$f^\ast:=\min\{f(x)\mid x\in S(\bar g)\}>0,$$
we have
$$f\in T\left(\bar g,cd^2\left(1+\left(d^2n^d\frac{\|f\|}{f^\ast}
  \right)^c\right)\right).$$
\end{theorem}
In this article, we will prove the following bound for Putinar's theorem.
\begin{theorem}\label{putinarbound}
For all $\bar g$ defining an archimedean quadratic module $M(\bar g)$ and
a set $\emptyset\neq S(\bar g)\subseteq(-1,1)^n$,
there is some $c\in\R_{>0}$ (depending on $\bar g$) such that for
all $f\in\rx$ of degree $d$ with
$$f^\ast:=\min\{f(x)\mid x\in S(\bar g)\}>0,$$
we have
$$f\in M\left(\bar g,c\exp\left(
\left(d^2n^d\frac{\|f\|}{f^\ast}\right)^c\right)\right).$$
\end{theorem}

In both theorems above, there have been made additional assumptions
compared to Schm\"udgen's and Putinar's original results. But these are not
very serious and have only been made to simplify the statements: For example,
if $S(\bar g)=\emptyset$, then $-1\in T(\bar g,k)$ for
some $k\in\N$ by Schm\"udgen's theorem. Therefore
$4f=(f+1)^2+(f-1)^2(-1)\in T(\bar g,2d+k)$ for each $f\in\rx$ of degree
$d\ge 0$. The other hypothesis that $S(\bar g)$ be contained in the open
hypercube $(-1,1)^n$ is only a matter of rescaling by a
linear (or affine linear) transformation on $\R^n$. For example, if $r>0$
is such that $S(\bar g)\subseteq(-r,r)^n$,
then Theorem \ref{schmuedgenbound} remains true with $\|f\|$ replaced by
$\|f(r\bar X)\|$. Here it is important to note that the property that
$M(\bar g)$ be archimedean is preserved under affine linear coordinate
changes. This is clear from Corollary \ref{ac}. Confer also the proof of
Proposition \ref{ease} below.

In both Theorem \ref{schmuedgenbound} and \ref{putinarbound}, the bound
depends on three parameters:
\begin{itemize}
\item The description $\bar g$ of the basic closed semialgebraic set,
\item the degree $d$ of $f$ and
\item a measure of how close $f$ comes to have a zero on $S(\bar g)$,
namely $\|f\|/f^*$.
\end{itemize}
The main difference between the two bounds is the exponential function
appearing in the degree bound for the quadratic module representation.
It is an open research problem whether this exponential function can be
avoided. It could even be possible that the same bound than for Schm\"udgen's
theorem holds also for Putinar's theorem. In view of the impact on the
convergence rate of Lasserre's optimization procedure (see Section
\ref{optimization} below), this question seems very interesting for
applications. Whereas the bound for the preordering representation cannot
be improved significantly (see \cite{ste}), this seems possible for the
quadratic module representation.

The dependance on the third parameter $\|f\|/f^*$ is consistent with the
fact that the condition $f^*>0$ cannot be weakened to $f^*\ge 0$ in neither
Schm\"udgen's nor Putinar's theorem. Under certain conditions (e.g., on the
derivatives of $f$), both theorems can however be extended to nonnegative
polynomials (see \cite{sch, mr2}). With the partially constructive approach
from \cite{sw4} to representation of nonnegative polynomials with zeros, one
might perhaps in the future gain bounds even for the case of nonnegative
polynomials which depend however on further data (for example the norm of
the Hessian at the zeros).

In special cases, Prestel had already proved the mere existence of a degree
bound for Putinar's Theorem depending on the three parameters described
above (see \cite[Section 8.4]{pd} and \cite{pre}). He used model theory and
valuation theory to get the existence of such a bound. But the only
information about the bound he gets (using G\"odel's theorem on the
completeness of first order logic) is that the bound is computable.

In contrast to this, our more constructive approach yields information in
what way the above bound depends on the two parameters $d$ and $\|f\|/f^*$.
The constant $c$ depends on the description
$\bar g$ of the semialgebraic set,
but no explicit formula is given.
For a concretely given $\bar g$, one could possibly determine a constant $c$
like in Theorems \ref{schmuedgenbound} and \ref{putinarbound} by a
very (probably too) tedious analysis of the proofs
(cf. \cite[Remark 10]{sw2}).

We conclude this introduction by considering the one variable case, i.e.,
$n=1$. Scheiderer showed in \cite[Corollary 3.4]{sch} that, in this case,
compactness of $S(\bar g)$ implies that $M(\bar g)=T(\bar g)$ (and therefore
$M(\bar g)$ is archimedean). Now the equality $M(\bar g)=T(\bar g)$ implies
in particular that ${\bar g}^\de\in M(\bar g)$ for all $\de\in\{0,1\}^m$.
As an easy consequence, we get that Theorem \ref{schmuedgenbound} remains
valid with $T$ replaced by $M$ in the case of univariate polynomials. The
bound in Theorem \ref{putinarbound} is thus far from being sharp in the
one variable case. As said above, in the multivariate case it is not known
if the bound can be improved considerably.

The rest of the paper is organized as follows. In the next section,
we use our result to investigate the accuracy of Lasserre's
``sums of squares
relaxations'' for optimization of polynomials. In Section \ref{proof}, we
give the proof of Theorem \ref{putinarbound}.

\section{Convergence rate of Lasserre's procedure}\label{optimization}

Consider the problem to compute (by a numerical procedure, i.e., up to
some prescribable error) the minimum
\begin{equation}\label{fstardef}
f^*:=\min\{f(x)\mid x\in S(\bar g)\}
\end{equation}
of a polynomial $f\in\rx$ on a non-empty basic closed semialgebraic set
$S(\bar g)$. In other words, you want to minimize
a polynomial under polynomial inequality constraints. When all the
polynomials involved are linear, i.e., of degree $\le 1$, this is
a linear optimization problem (a linear program) and there are very efficient
algorithms
to solve this problem. For general polynomials this problem gets very hard.
It is therefore a common approach to solve a much easier related problem, a so
called relaxation, namely to compute for $k\in\N$,
\begin{equation}\label{fkdef}
f_k^*:=\sup\{a\in\R\mid f-a\in M(\bar g,k)\}\in\R\cup\{-\infty\}
\end{equation}
which is clearly a lower bound of $f^*$. The problem of finding $f_k^*$ can
be written as a semidefinite program whose size gets bigger when $k$ grows
(see the references below).
Semidefinite programming is a well-known generalization of linear
programming for which very efficient algorithms exist (see for example
\cite{tod}). One can now solve
a sequence of larger and larger semidefinite programs in order to get
tighter and tighter lower bounds for $f^*$. Lasserre \cite{las} was the first
to interpret Putinar's theorem as a convergence result.

Indeed, it is easy to see that Putinar's theorem just says that the ascending
sequence $(f_k^*)_{k\in\N}$ converges to $f^*$ under the condition that
$M(\bar g)$ be archimedean.
In this section, we will interpret our bound for Putinar's Positivstellensatz
as a result about the speed of convergence of this sequence.

For an introduction to the interplay of semidefinite programming, sums of
squares, optimization of polynomials and results about positive polynomials,
we refer to \cite{las,mr1,sw1} (with special regard to Putinar's
Positivstellensatz) and \cite{jl,dnp,nds,ps}. There are several software
tools which translate the problem of computing $f_k^*$ into a semidefinite
program and call a semidefinite programming solver. See \cite{hl,kkw,löf,sos}.

The following technical lemma will also be needed in Section \ref{proof}.

\begin{lemma}\label{corest}
For any polynomial $f\in\rx$ of degree $d\ge 1$ and all $x\in[-1,1]^n$,
$$|f(x)|\le 2dn^d\|f\|.$$
\end{lemma}

\begin{proof}
Writing $f=\sum_\al a_\al\binom{|\al|}\al\x^\al$ ($a_\al\in\R$), we have
$\|f\|=\max_\al|a_\al|$ and
\begin{equation*}
|f(x)|=\left|\sum_\al a_\al\binom{|\al|}\al x_1^{\al_1}\dotsm x_n^{\al_n}
\right|
\le \sum_\al|a_\al|\binom{|\al|}\al|x_1|^{\al_1}\dotsm|x_n|^{\al_n}.
\end{equation*}
for all $x\in[-1,1]^n$.
Using that $|a_\al|\le\|f\|$ and $|x_i|\le 1$, the multinomial identity
now shows that $|f(x)|\le\|f\|\sum_{k=0}^dn^k\le(d+1)n^d\|f\|\le 2dn^d\|f\|$.
\end{proof}

Now we are ready to prove the main theorem of this section.

\begin{theorem}\label{howmuchadd}
For all polynomials $\bar g$ defining an archimedean quadratic module
$M(\bar g)$ and a set $\emptyset\neq S(\bar g)\subseteq (-1,1)^n$, there is
some $c>0$ (depending on $\bar g$) such that for all $f\in\rx$ of degree
$d$ with minimum $f^*$ on $S$ and for all integers
$k>c\exp((2d^2n^d)^c)$, we have
$$(f-f^*)+\frac{6d^3n^{2d}\|f\|}{\sqrt[c]{\log\frac kc}}\in M(\bar g,k)$$
and hence
$$0\le f^*-f_k^*\le \frac{6d^3n^{2d}\|f\|}{\sqrt[c]{\log\frac kc}}$$
where $f_k^*$ is defined as in (\ref{fkdef}).
\end{theorem}

\begin{proof}
Given $\bar g$, we choose $c>0$ like in Theorem \ref{putinarbound}. Now let
$f\in\rx$ be of degree $d$ with minimum $f^*$ on $S$ and
\begin{equation}\label{kbig}
k>c\exp((2d^2n^d)^c)
\end{equation}
be an integer. The case $d=0$ is trivial. We assume therefore $d\ge 1$.
Note that $k>c$ and hence $\log(k/c)>0$. Setting
\begin{equation}\label{mustadd}
a:=\frac{6d^3n^{2d}\|f\|}{\sqrt[c]{\log\frac kc}},
\end{equation}
all we have to prove is $h:=f-f^*+a\in M(\bar g,k)$ because the second claim
follows from this. By our choice of $c$ and the observation $\deg h=\deg f=d$,
it is enough to show that
$$c\exp\left(\left(d^2n^d\frac{\|h\|}a\right)^c\right)\le k,$$
or equivalently
$$d^2n^d\|h\|\le a\sqrt[c]{\log\frac kc}=6d^3n^{2d}\|f\|.$$
Observing that $\|h\|\le\|f\|+|f^*|+a$, it suffices to show that
$$\|f\|+|f^*|+a\le 6dn^d\|f\|.$$
Lemma \ref{corest} tells us that $|f^*|\le 2dn^d\|f\|$ and we are thus reduced
to verify that
$$a\le(4dn^d-1)\|f\|$$
which is by (\ref{mustadd}) equivalent to
$$6d^3n^{2d}\le (4dn^d-1)\sqrt[c]{\log\frac kc}.$$
By (\ref{kbig}), it is finally enough to check that
$6d^3n^{2d}\le(4dn^d-1)(2d^2n^d)$.
\end{proof}

As already said in the introduction, the hypothesis that $S(\bar g)$ is
contained in the open unit hypercube is just a technicality to avoid that
the bound gets even more complicated. In fact, if one does not insist on all
the information given in Theorem \ref{howmuchadd}, one gets a corollary
which is easy to remember and still gives the most important part of
information.

\begin{corollary}\label{ease}
Suppose $M(\bar g)$ is archimedean, $S(\bar g)\neq\emptyset$ and $f\in\rx$.
There is
\begin{itemize}
\item a constant $c>0$ depending only on $\bar g$ and
\item a constant $c'>0$ depending on $\bar g$ and $f$
\end{itemize}
such that for $f^*$ and $f_k^*$ as defined in (\ref{fstardef}) and
(\ref{fkdef}),
$$0\le f^*-f_k^*\le\frac{c'}{\sqrt[c]{\log\frac kc}}\qquad
  \text{for all large $k\in\N$.}$$
\end{corollary}

\begin{proof}
Without loss of generality, assume $f\neq 0$. Set $d:=\deg f$. Since
$M(\bar g)$ is
archimedean, $S(\bar g)$ is compact. We can hence choose a rescaling factor
$r>0$ depending only on $\bar g$ such that $S(\bar g(r\x))\subseteq(-1,1)^n$.
Here $\bar g(r\x)$ denotes the tuple of rescaled polynomials $g_i(r\x)$.
Now Theorem \ref{howmuchadd} applied to $g(r\x)$ instead of $\bar g$ yields
$c>0$ that will together with $c':= 6d^3n^{2d}\|f(rX)\|$ have the desired
properties by simple scaling arguments.
\end{proof}

\begin{remark}
The bound on the difference $f^*-f_k^*$ presented in this section is much worse
than the corresponding one presented in \cite[Section 2]{sw2} which is based
on preordering representations (i.e., where $f_k^*$ would be defined using
$T(\bar g)$ instead of $M(\bar g)$). This raises the question whether it is
after all not such a bad thing to use preordering (instead of quadratic
module) representations for
optimization though they involve the $2^m$ products ${\bar g}^\de$ letting
the semidefinite programs get huge when $m$ is not small. However, it is not
known if Theorem \ref{howmuchadd} holds perhaps even with the bound from
\cite[Theorem 4]{sw2}. Compare also \cite[Remark 5]{sw2}.
\end{remark}

\section{The proof}\label{proof}

In this section, we give the proof of Theorem \ref{putinarbound}. The three
main ingredients are
\begin{itemize}
\item the bound for Schm\"udgen's theorem
presented in Theorem \ref{schmuedgenbound} above,
\item ideas from the (to some extent constructive) proof of Putinar's theorem
in \cite[Section 2]{sw3} and
\item the \L ojasiewicz inequality from semialgebraic geometry.
\end{itemize}
We start with some simple facts from calculus.

\begin{lemma}\label{lest}
If $0\neq f\in\rx$ has degree $d$, then
$$|f(x)-f(y)|\le\|x-y\|d^2n^{d-1}\sqrt n\|f\|$$
for all $x,y\in [-1,1]^n$.
\end{lemma}

\begin{proof}
Denoting by $Df$ the derivative of $f$, by the mean value theorem, it is
enough to show that
\begin{equation}\label{derred}
|Df(x)(e)|\le d^2n^{d-1}\sqrt n\|f\|
\end{equation}
for all $x\in [-1,1]^n$ and $e\in\R^n$ with $\|e\|=1$.
A small computation (compare the proof of Lemma \ref{corest}) shows that
$$\left|\frac{\partial f(x)}{\partial x_i}\right|\le
  \|f\|\sum_{k=1}^dk(|x_1|+\dots+|x_n|)^{k-1}
  \le\|f\|\sum_{k=1}^dkn^{k-1}\le\|f\|d^2n^{d-1},$$
from which we conclude for all $x\in [-1,1]^n$ and $e\in\R^n$ with $\|e\|=1$,
$$|Df(x)(e)|=\left|\sum_{i=1}^n\frac{\partial f(x)}{\partial x_i}e_i\right|
  \le\sum_{i=1}^n\left|\frac{\partial f(x)}{\partial x_i}\right|\cdot|e_i|
  \le\|f\|d^2n^{d-1}\sum_{i=1}^n|e_i|.
$$
Because for a vector $e$ on the unit sphere in $\R^n$, $\sum_{i=1}^n|e_i|$
can reach at most $\sqrt n$, this implies (\ref{derred}).
\end{proof}

\begin{remark}\label{calculus}
For all $k\in\N$ and $y\in [0,1]$,
$$(y-1)^{2k}y\le\frac 1{2k+1}.$$
\end{remark}

The next lemma is a version of \cite[Lemma 2.3]{sw3} caring about
complexity issues. In \cite[Lemma 2.3]{sw3}, it is shown that, if
$C\subseteq\R^n$ is any compact set, $g_i\le 1$ on $C$ for all $i$ and
$f\in\rx$ is a polynomial with $f>0$ on $S(\bar g)$, then there exists
$\la\ge 0$ such that for all sufficiently large $k\in\N$,
\begin{equation}\label{clifting}
f-\la\sum_{i=1}^m(g_i-1)^{2k}g_i>0\qquad\text{on $C$.}
\end{equation}
The idea is that, if you want to show that $f\in M(\bar g)$, you first
subtract another polynomial from $f$ which lies obviously in $M(\bar g)$
such that the difference can be proved to lie in $M(\bar g)$ as well.
This other polynomial must necessarily be nonnegative on $S(\bar g)$ but
it should take on only very small values on $S(\bar g)$ so that the difference
is still positive on $S(\bar g)$.
On the region where you are outside and not too far away from $S(\bar g)$,
the polynomial you subtract should take large negative values
so that the difference gets positive on this region outside of $S(\bar g)$
(where $f$ itself might be negative). The hope is that the difference satisfies
an improved positivity condition which will help us to show that it lies in
$M(\bar g)$. To understand the lemma, it is helpful to observe that the
pointwise limit for $k\to\infty$ of this difference, which is the left hand
side of (\ref{claim}), is $f$ on $S(\bar g)$ and $\infty$ outside of
$S(\bar g)$.

\begin{lemma}\label{lifting}
For all $\bar g$ such that $S:=S(\bar g)\cap[-1,1]^n\neq\emptyset$ and
$g_i\le 1$ on $[-1,1]^n$, there are $c_0,c_1,c_2>0$ with the following
property:

For all polynomials $f\in\rx$ of degree $d$ with minimum $f^*>0$ on $S$,
if we set
\begin{equation}\label{defal}
L:=d^2n^{d-1}\frac{\|f\|}{f^*},\qquad
\la:=c_1d^2n^{d-1}\|f\|L^{c_2}
\end{equation}
and if $k\in\N$ satisfies
\begin{equation}\label{kcond}
2k+1\ge c_0(1+L^{c_0}),
\end{equation}
then the inequality
\begin{equation}\label{claim}
f-\la\sum_{i=1}^m(g_i-1)^{2k}g_i\ge\frac{f^*}2
\end{equation}
holds on $[-1,1]^n$.
\end{lemma}

\begin{proof}
By the \L ojasiewicz inequality for semialgebraic functions
(Corollary 2.6.7 in \cite{bcr}), we can choose
$c_2,c_3>0$ such that
\begin{equation}\label{lojasiewicz}
\dist(x,S)^{c_2}\le -c_3\min\{g_1(x),\dots,g_m(x),0\}
\end{equation}
for all $x\in[-1,1]^n$ where $\dist(x,S)$ denotes the distance of $x$ to $S$.
Set
\begin{align}
\label{defc4}c_4&:=c_3(4n)^{c_2},\\
\label{defc1}c_1&:=4nc_4
\end{align}
and choose $c_0\in\N$ big enough to guarantee that
\begin{align}
\label{c0cond2}
c_0(1+r^{c_0})&\ge 2(m-1)c_4r^{c_2}\qquad\text{and}\\
\label{c0cond3}
c_0(1+r^{c_0})&\ge 4mc_1r^{c_2+1}
\end{align}
for all $r\ge 0$. Now suppose $f\in\rx$ is of degree $d$ with
minimum $f^*>0$ on $S$ and consider the set
$$A:=\left\{x\in[-1,1]^n\mid f(x)\le\frac 34f^*\right\}.$$
By Lemma \ref{lest}, we get for all $x\in A$ and $y\in S$
$$\frac{f^*}4\le f(y)-f(x)\le\|x-y\|d^2n^{d-1}\sqrt n\|f\|\le
  \|x-y\|d^2n^d\|f\|.$$
Since this is valid for arbitrary $y\in S$, it holds that
$$\frac{f^*}{4d^2n^d\|f\|}\le\dist(x,S)$$
for all $x\in A$. We combine this now with (\ref{lojasiewicz}) and get
$$\min\{g_1(x),\dots,g_m(x)\}\le -\frac 1{c_3}
  \left(\frac{f^*}{4d^2n^d\|f\|}\right)^{c_2}
$$
for $x\in A$. We have omitted the argument $0$ in the minimum which is here
redundant because of $A\cap S=\emptyset$. By setting
\begin{equation}\label{defdelta}
\de:=\frac 1{c_4L^{c_2}}>0,
\end{equation}
where we define $L$ like in (\ref{defal}), and having a look at
(\ref{defc4}), we can rewrite this as
\begin{equation}\label{rewritten}
\min\{g_1(x),\dots,g_m(x)\}\le -\de.
\end{equation}
Define $\la$ and $k$ like in (\ref{defal}) and (\ref{kcond}). For later use,
we note
\begin{equation}\label{lal}
\la=c_1L^{c_2+1}f^*.
\end{equation}
We claim now that
\begin{align}
\label{cond1}f+\frac{\la\de}2&\ge\frac{f^*}2\text{\ on\ } [-1,1]^n,\\
\label{cond2}\frac\de 2&\ge\frac{m-1}{2k+1}\qquad\text{and}\\
\label{cond3}\frac{f^*}4&\ge\frac{\la m}{2k+1}.
\end{align}

Let us prove these claims.
If we choose in Lemma \ref{lest} for $y$ a minimizer of $f$ on $S$, we
obtain
$$|f(x)-f^*|\le\diam([-1,1]^n)d^2n^{d-1}\sqrt n\|f\|
  =2\sqrt nd^2n^{d-1}\sqrt n\|f\|=2d^2n^d\|f\|
$$
for all $x\in [-1,1]^n$, noting that the diameter of $[-1,1]^n$ is $2\sqrt n$.
In particular, we observe
$$f\ge f^*-2d^2n^d\|f\|\ge\frac{f^*}2-2d^2n^d\|f\|\qquad
  \text{on\ }[-1,1]^n.$$
Together with the equation
$$\frac{\la\delta}2=2d^2n^d\|f\|,$$
which is clear from (\ref{defal}), (\ref{defc1}) and (\ref{defdelta}),
this yields (\ref{cond1}). Using (\ref{kcond}), (\ref{c0cond2}) and
(\ref{defdelta}), we see that
$$(2k+1)\de\ge c_0(1+L^{c_0})\de\ge 2(m-1)c_4L^{c_2}\de=2(m-1)$$
which is nothing else than (\ref{cond2}). Finally, we exploit (\ref{kcond}),
(\ref{c0cond3}) and (\ref{lal}), to see that
$$(2k+1)f^*\ge c_0(1+L^{c_0})f^*\ge 4mc_1L^{c_2+1}f^*=4m\la,$$
i.e., (\ref{cond3}) holds.

Now (\ref{cond1}), (\ref{cond2}) and (\ref{cond3})
will enable us to show our claim (\ref{claim}). If $x\in A$, then in the sum
\begin{equation}\label{sumlabel}
\sum_{i=1}^m(g_i(x)-1)^{2k}g_i(x)
\end{equation}
at most $m-1$ summands are nonnegative. By Remark \ref{calculus}, these
nonnegative summands add up to at most $(m-1)/(2k+1)$. At least one
summand is negative, even $\le-\de$ by (\ref{rewritten}). All in all, if we
evaluate the left hand side of our claim (\ref{claim}) in a point $x\in A$,
then it is
$$\ge f(x)-\la\frac{m-1}{2k+1}+\la\de\ge
  \underbrace{f(x)+\frac{\la\de}2}_{\ge\frac{f^*}2\text{\ by (\ref{cond1})}}
  +\la
  \underbrace{\left(\frac\de 2-\frac{m-1}{2k+1}\right)}_
                    {\ge 0\text{\ by (\ref{cond2})}
             }
  \ge\frac{f^*}2.
$$
When we evaluate it in a point $x\in [-1,1]^n\setminus A$, all summands of
the sum (\ref{sumlabel}) might happen to be nonnegative. Again by Remark
\ref{calculus}, they add up to at most $m/(2k+1)$. But at the same
time,
the definition of $A$ gives us a good lower bound on $f(x)$ so that the
result is
$$\ge \frac 34{f^*}-\la\frac m{2k+1}\ge \frac{f^*}2+
  \underbrace{\frac{f^*}4-\frac{\la m}{2k+1}}_{\ge 0\text{\ by (\ref{cond3})}}
  \ge\frac{f^*}2.$$
\end{proof}

\begin{proposition}\label{normprop}
If $p,q\in\rx$ are both homogeneous (i.e., all of their
respective monomials have the same degree), then $\|pq\|\le\|p\|\|q\|$.
For arbitrary
$s\in\N$ and polynomials $0\neq p_1,\dots,p_s\in\rx$, we have
$$\|p_1\dotsm p_s\|\le (1+\deg p_1)\dotsm(1+\deg p_s)\|p_1\|\dotsm\|p_s\|.$$
\end{proposition}

\begin{proof}
The statement for homogeneous $p$ and $q$ can be found in
\cite[Lemma 8]{sw2}. The second claim follows from this by writing each
$p_i$ as a sum
$p_i=\sum_kp_{ik}$ of homogeneous degree $k$ polynomials $p_{ik}$. Multiply
the $p_i$
by distributing out all such sums and apply the triangle inequality to the
sum which arises in this way. Then use
$$\|p_{1{k_1}}\dotsm p_{s{k_s}}\|\le\|p_{1{k_1}}\|\dotsm\|p_{s{k_s}}\|
\le\|p_1\|\dotsm\|p_s\|.$$
Now factor out $\|p_1\|\dotsm\|p_s\|$ and recombine the terms of the sum
which now are all constant $1$.
\end{proof}

\begin{lemma}\label{explemma}
For all $c_1,c_2,c_3 >0$, there is $c>0$ such that
$$c_1\exp(c_2 r^{c_3})\le c\exp(r^c)\qquad\text{for all $r\ge 0$}.$$
\end{lemma}

\begin{proof}
Choose any $c\ge c_1\exp(c_22^{c_3})$ such that
$c_3\le c/2$ and $c_2\le 2^{c/2}$. Then for $r\in [0,2]$,
$$c_1\exp(c_2 r^{c_3})\le c_1\exp(c_2 2^{c_3})\le c\le c\exp(r^c)$$
and for $r\ge 2$ (observing that $c_1\le c$),
$$c_1\exp(c_2 r^{c_3})\le c\exp(2^{c/2}r^{c/2})\le c\exp(r^c).$$
\end{proof}

We resume the discussion before Lemma \ref{lifting}. With regard to
(\ref{claim}), we can for the moment concentrate on polynomials positive
on the hypercube $[-1,1]^n$. If this hypercube could be described by a single
polynomial inequality, i.e., if we had $[-1,1]^n=S(p)$ for some $p\in\rx$,
then the idea would be to apply the bound for
Schm\"udgen's Positivstellensatz now.
The clue is here that $p$ is a single polynomial and hence preordering
and quadratic module representations are the same, i.e., $T(p)=M(p)$.
The following lemma works around the fact that $[-1,1]^n=S(p)$ can
only happen when $n=1$. We round the edges of the hypercube.

\begin{lemma}\label{round}
Let $S\subseteq (-1,1)^n$ be compact. Then
$1-\frac 1d-(X_1^{2d}+\dots X_n^{2d})>0$ on $S$ for all sufficiently large
$d\in\N$.
\end{lemma}

\begin{proof}
Consider for each $1\le d\in\N$ the set
$$A_d:=\left\{x\in S\mid x_1^{2d}+\dots+x_n^{2d}\ge 1-\frac 1d\right\}.$$
This gives a decreasing sequence
$A_1\supseteq A_2\supseteq A_3\supseteq\dots$
of compact sets whose intersection $\cap_{d=1}^\infty A_d$ is empty by
calculus.
By compactness, a finite subintersection is empty, i.e.,
$A_d=\emptyset$ for all large $d\in\N$.
\end{proof}

Note that in the proof of Putinar's theorem in \cite[Section 2]{sw3} where we
were not interested in complexity, a different approach has been taken.
Condition (\ref{clifting}) has been established for a polyhedron $C$ which is
even bigger than the hypercube, so big that preordering representations
certifying nonnegativity on $C$ can be turned into quadratic module
representations certifying nonnegativity on the hypercube. The advantage was
that we could use P\'olya's theorem \cite{pol} which is much more elementary
than Schm\"udgen's theorem. Despite the existence of the effective version
\cite{pr} of that
theorem of P\'olya, it seems that establishing positivity on such a big
polyhedron $C$ is too expensive from the complexity point of view. Though
it is not so nice, we therefore work here with a rounded hypercube and Theorem
\ref{schmuedgenbound} instead.

We finally attack the proof of Theorem \ref{putinarbound}.

\begin{proof}[Proof of Theorem \ref{putinarbound}]
By a simple scaling argument, we may assume that $\|g_i\|\le 1$ and
$g_i\le 1$ on $[-1,1]$ for all $i$.
According to Lemma \ref{round}, we can choose $d_0\in\N$ such that
$$p:=1-\frac 1{d_0}-(X_1^{2d}+\dots+X_n^{2d})>0\text{\ on\ }S(\bar g).$$
By Putinar's Theorem \ref{putinar}, we have $p\in M(\bar g)$
and therefore
\begin{equation}\label{solid}
p\in M(\bar g,d_1)
\end{equation}
for some $d_1\in\N$.
Choose $d_2\in\N$ such that
\begin{equation}\label{d2def}
1+\deg g_i\le d_2\qquad\text{for all $i\in\{1,\dots,m\}$.}
\end{equation}
Now we choose $c_0,c_1,c_2$ like in Lemma \ref{lifting}, define
$L$ and $\la$ like in (\ref{defal}) and choose the smallest $k\in\N$
satisfying (\ref{kcond}). Then
\begin{equation}\label{twok}
2k+1\le c_0(1+L^{c_0})+2.
\end{equation}
Let $c_3\ge 1$ denote the constant existing by Theorem \ref{schmuedgenbound}
(which is there called $c$ and gives the bound for preordering
representations of polynomials positive on $S(\bar g)$).
Using Lemma \ref{explemma}, it is easy to see
that we can choose $c_4,c_5,c_6,c_7,c\ge 0$ satisfying
\begin{align}
\label{choicec4}
c_32^{c_3}r^{2+2c_3}n^{c_3r}&
\le c_4(\exp(c_4r))\\
\label{choicec5}
2r+2c_1r^{c_2+1}d_2^{r(1+r^{c_0})+1}&
\le c_5\exp(r^{c_5})\\
\label{choicec6}
c_4\exp(2c_4d_2r(1+r^{c_0}+3))&
\le c_6\exp(r^{c_6})\\
\label{choicec7}
c_5^{c_3}c_6\exp(c_3r^{c_5}+r^{c_6})&
\le c_7\exp(r^{c_7})\\
\label{choicec}
c_7\exp(r^{c_7})+d_1&
\le c\exp(r^c)
\end{align}
for all $r\ge 0$.
Now let $f\in\rx$ be a polynomial of degree $d\ge 1$ with
$$f^\ast:=\min\{f(x)\mid x\in S(\bar g)\}>0.$$
We are going to apply Theorem \ref{schmuedgenbound} to
\begin{equation*}
h:=f-\la\sum_{i=1}^m(g_i-1)^{2k}g_i.
\end{equation*}
By Lemma \ref{lifting}, (\ref{claim}) holds for this polynomial,
in particular
\begin{equation}\label{hstar}
h^\ast:=\min\{h(x)\mid x\in S(p)\}
\ge\frac{f^\ast}2.
\end{equation}
By Proposition \ref{normprop} and the definition of $d_2$ in (\ref{d2def}),
\begin{align}
\label{hnorm}
\|h\|&\le\|f\|+\la d_2^{2k+1}\\
\label{dhbound}
\deg h&\le\max\{d,(2k+1)d_2,1\}=:d_h.
\end{align}
By Theorem \ref{schmuedgenbound} (respectively the above choice of
$c_3\ge 1$),
we get
\begin{equation}\label{hbound}
h\in T(p,k_h)
\qquad\text{where\ }k_h:=c_3d_h^2\left(1+d_h^2n^{d_h}
\frac{\|h\|}{h^*}\right)^{c_3}.
\end{equation}
Note that $\|h\|/h^*\ge 1$ since $0<h^*\le h(0)\le\|h\|$. We use this
to simplify the degree bound in (\ref{hbound}). Obviously
\begin{multline}\label{khbound}
k_h\le
c_3d_h^2\left(2d_h^2n^{d_h}\frac{\|h\|}{h^*}\right)^{c_3}\\
\le c_32^{c_3}d_h^{2+2c_3}n^{c_3d_h}\left(\frac{\|h\|}{h^*}\right)^{c_3}
\le c_4\exp(c_4d_h)\left(\frac{\|h\|}{h^*}\right)^{c_3}
\end{multline}
by choice of $c_4$ in (\ref{choicec4}).
Moreover, we have
\begin{multline}\label{hhbound}
\frac{\|h\|}{h^*}\le\frac 2{f^*}(\|f\|+\la d_2^{2k+1})
=2\frac{\|f\|}{f^*}+2c_1d_2^{2k+1}L^{c_2+1}\\
\le 2L+2c_1d_2^{2k+1}L^{c_2+1}
= 2L+2c_1L^{c_2+1}d_2^{c_0(1+L^{c_0})+1}
\le c_5\exp(L^{c_5})
\end{multline}
by (\ref{hnorm}), (\ref{hstar}), (\ref{twok}), (\ref{lal}) and by the choice
of $c_5$ in (\ref{choicec5}).
It follows that
\begin{align*}
d_h&\le d(2k+2)d_2&\text{(by (\ref{dhbound}))}\\
&\le d(c_0(1+L^{c_0})+3)d_2&\text{(by (\ref{twok}))}\\
&\le 2d_2d^2n^d\frac{\|f\|}{2dn^d\|f\|}(c_0(1+L^{c_0})+3)\\
&\le 2d_2d^2n^d\frac{\|f\|}{f^*}(c_0(1+L^{c_0})+3)
&\text{(by Lemma \ref{corest})}\\
&\le 2d_2nL(c_0(1+(nL)^{c_0}+3))&\text{(by (\ref{defal}))}
\end{align*}
and therefore
\begin{equation}\label{expdhbound}
c_4\exp(c_4d_h)\le c_6\exp((nL)^{c_6})
\end{equation}
for the constant $c_6$ chosen in (\ref{choicec6}). We now get
\begin{align*}
k_h&\le c_4\exp(c_4d_h)\left(\frac{\|h\|}{h^*}\right)^{c_3}
&\text{(by (\ref{khbound}))}\\
&\le c_6\exp((nL)^{c_6})(c_5\exp(L^{c_5}))^{c_3}
&\text{(by (\ref{expdhbound}) and (\ref{hhbound}))}\\
&=c_5^{c_3}c_6\exp(c_3(nL)^{c_5}+(nL)^{c_6})\\
&\le c_7\exp((nL)^{c_7})
&(\text{by choice of $c_7$ in (\ref{choicec7})).}
\end{align*}
Combining this with (\ref{hbound}) and (\ref{solid}), i.e.,
$$h\in T(p,c_7\exp((nL)^{c_7}))\qquad\text{and}\qquad
  p\in M(\bar g,d_1),$$
yields (by composing corresponding representations)
$$h\in M(\bar g,c\exp((nL)^c))$$
according to the choice of $c$ in (\ref{choicec}).
Finally, we have that
$$f=h+\la\sum_{i=1}^m(g_i-1)^{2k}g_i\in M(\bar g,c\exp((nL)^c))$$
since
$$\deg((g_i-1)^{2k}g_i)\le d_h\le k_h\le c_7\exp((nL)^{c_7})\le
c\exp((nL)^c)$$
by choice
of $d_2$ in (\ref{d2def}), $d_h$ in (\ref{dhbound}), $k_h$ in (\ref{hbound})
and c in (\ref{choicec}).
\end{proof}

\section*{Acknowledgments}

The authors would like to thank James Demmel, Vicki Powers,
Mihai Putinar and Bernd Sturmfels for the their fruitful suggestions
helping to improve this paper.


\begin{thebibliography}{KKW}

\bibitem[BCR]{bcr} J. Bochnak, M. Coste, M.-F. Roy:
Real algebraic geometry,
Ergebnisse der Mathematik und ihrer Grenzgebiete {\bf 36}, Berlin:
Springer (1998)

\bibitem[DNP]{dnp} J. Demmel, J. Nie and V. Powers:
Representations of positive polynomials on non-compact semialgebraic
sets via
KKT ideals, to appear in J. Pure Appl. Algebra\\
\url{http://math.berkeley.edu/~njw/}

\bibitem[HL]{hl} D. Henrion and J. Lasserre:
GloptiPoly: Global Optimization over Polynomials with Matlab and SeDuMi\\
\url{http://www.laas.fr/~henrion/software/gloptipoly/}

\bibitem[JL]{jl} D. Jibetean and M. Laurent:
Semidefinite approximations for global unconstrained polynomial optimization,
SIAM J. Optim. {\bf 16}, No. 2, 490--514 (2005)

\bibitem[JP]{jp} T. Jacobi, A. Prestel:
Distinguished representations of strictly positive polynomials,
J. Reine Angew. Math. {\bf 532}, 223--235 (2001)

\bibitem[Las]{las} J. Lasserre:
Global optimization with polynomials and the problem of moments,
SIAM J. Optim. {\bf 11}, No. 3, 796--817 (2001)

\bibitem[L\"of]{löf} J. L\"ofberg:
YALMIP: A MATLAB toolbox for rapid prototyping of optimization problems\\
\url{http://control.ee.ethz.ch/~joloef/yalmip.php}

\bibitem[Mr1]{mr1} M. Marshall:
Optimization of polynomial functions,
Can. Math. Bull. {\bf 46}, No. 4, 575--587 (2003)

\bibitem[Mr2]{mr2} M. Marshall:
Representation of non-negative polynomials with finitely many zeros,
to appear in Annales de la Facult\'e des Sciences de Toulouse\\
\url{http://math.usask.ca/~marshall/}


\bibitem[NDS]{nds} J. Nie, J. Demmel, and B. Sturmfels:
Minimizing polynomials via sum of squares over the gradient ideal,
Math. Program. {\bf 106}, No. 3 (A), 587--606 (2006)

\bibitem[PD]{pd} A. Prestel, C. Delzell:
Positive polynomials,
Springer Monographs in Mathematics, Berlin: Springer (2001)

\bibitem[P\'ol]{pol} G. P\'olya:
\"Uber positive Darstellung von Polynomen,
Vierteljahresschrift der Naturforschenden Gesellschaft in Z\"urich {\bf 73}
(1928), 141--145, reprinted in: Collected Papers, Volume 2, 309--313,
Cambridge: MIT Press (1974)

\bibitem[PR]{pr} V. Powers, B. Reznick:
A new bound for P\'olya's Theorem with applications to polynomials
positive on polyhedra,
J. Pure Appl. Algebra {\bf 164}, No. 1--2, 221--229 (2001)

\bibitem[Pre]{pre} A. Prestel:
Bounds for representations of polynomials positive on compact semi-algebraic
sets,
Fields Inst. Commun. {\bf 32}, 253--260 (2002)

\bibitem[PS]{ps} P. Parrilo, B. Sturmfels:
Minimizing polynomial functions,
DIMACS Series in Discrete Mathematics and Theoretical Computer Science
{\bf 60}, 83--100 (2003)

\bibitem[Put]{put} M. Putinar:
Positive polynomials on compact semi-algebraic sets,
Indiana Univ. Math. J. {\bf 42}, No. 3, 969--984 (1993)

\bibitem[Sch]{sch} C. Scheiderer:
Distinguished representations of non-negative polynomials,
J. Algebra {\bf 289}, No. 2, 558--573 (2005)

\bibitem[Smn]{smn} K. Schm\"udgen:
The $K$-moment problem for compact semi-algebraic sets,
Math. Ann. {\bf 289}, No. 2, 203--206 (1991)

\bibitem[Sw1]{sw1} M. Schweighofer:
An algorithmic approach to Schm\"udgen's Positivstellensatz,
J. Pure Appl. Algebra {\bf 166}, No. 3, 307--319 (2002)

\bibitem[Sw2]{sw2} M. Schweighofer:
On the complexity of Schm\"udgen's Positivstellensatz,
Journal of Complexity {\bf 20}, 529-543 (2004)

\bibitem[Sw3]{sw3} M. Schweighofer:
Optimization of polynomials on compact semialgebraic sets,
SIAM Journal on Optimization {\bf 15}, No. 3, 805--825 (2005)

\bibitem[Sw4]{sw4} M. Schweighofer:
Certificates for nonnegativity of polynomials with zeros on compact
semialgebraic sets,
Manuscripta Mathematica {\bf 117}, No. 4, 407 - 428 (2005)

\bibitem[SoS]{sos} S. Prajna, A. Papachristodoulou, P. Seiler, P. Parrilo:
SOSTOOLS: Sum of Squares Optimization Toolbox for MATLAB\\
\url{http://www.cds.caltech.edu/sostools/}

\bibitem[Ste]{ste} G. Stengle:
Complexity estimates for the Schm\"udgen Positivstellensatz,
J. Complexity {\bf 12}, No. 2, 167--174 (1996)

\bibitem[Tod]{tod} M. Todd: Semidefinite Optimization,
Acta Numerica {\bf 10}, 515-560 (2001)

\bibitem[KKW]{kkw} M. Kojima, S. Kim, H. Waki:
Sparsity in sums of squares of polynomials,
Math. Program. {\bf 103}, No. 1 (A), 45--62 (2005)\\
\url{http://www.is.titech.ac.jp/~kojima/SparsePOP/}

\end{thebibliography}
\end{document}